\documentclass[11pt, a4paper]{amsart}

\usepackage{amsmath,amssymb,amsthm}
\usepackage{hyperref}
\usepackage[dvipsnames,usenames]{color}
\usepackage[latin1]{inputenc}
\usepackage{graphicx}
\usepackage{psfrag}
\usepackage{color}

\newtheorem{guia}{}
\newtheorem{rem}{}

\newtheorem{teorema}[guia]{Theorem}

\newtheorem{coro}[guia]{Corollary}
\newtheorem{lema}[guia]{Lemma}

\newtheorem{obs}[rem]{\it Remark}

\newcommand{\al}{\alpha}

\newcommand{\De}{\Delta}

\newcommand{\g}{\gamma}

\newcommand{\la}{\lambda}

\newcommand{\Om}{\Omega}

\newcommand{\p}{\partial}

\newcommand{\R}{\mathbb R}

\newcommand{\ds}{\displaystyle}

\begin{document}

\title[Symmetry of large solutions]
{\bf Symmetry of large solutions for semilinear elliptic equations in a ball}

\author[C. Cort\'azar, M. Elgueta and J. Garc\'{\i}a-Meli\'{a}n]
{Carmen Cort\'azar, Manuel Elgueta\\ and Jorge
Garc\'{\i}a-Meli\'{a}n}

\date{}

\address{C. Cort\'azar and M. Elgueta \hfill\break\indent
Departamento de Matem\'aticas \hfill\break\indent Facultad de
Matem\'aticas, Pontificia Universidad Cat\'olica de Chile
\hfill\break\indent Casilla 306, correo 22 -- Santiago, CHILE}
\email{{\tt ccortaza@mat.puc.cl, melgueta@mat.puc.cl}}

\address{J. Garc\'{\i}a-Meli\'{a}n \hfill\break\indent
Departamento de An\'{a}lisis Matem\'{a}tico, Universidad de La
Laguna \hfill \break \indent C/. Astrof\'{\i}sico Francisco
S\'{a}nchez s/n, 38200 -- La Laguna, SPAIN\hfill\break\indent
{\rm and} \hfill\break
\indent Instituto Universitario de Estudios Avanzados (IUdEA) en F\'{\i}sica
At\'omica,\hfill\break\indent Molecular y Fot\'onica,
Universidad de La Laguna\hfill\break\indent C/. Astrof\'{\i}sico Francisco
S\'{a}nchez s/n, 38200 -- La Laguna , SPAIN} \email{{\tt
jjgarmel@ull.es}}

%%%%%%%%0.Abstract

\begin{abstract}
In this work we consider the boundary blow-up problem 
$$
\left\{
\begin{array}{ll}
\De u = f(u) & \hbox{in } B\\
\ \ u=+\infty & \hbox{on }\p B
\end{array}
\right.
$$
where $B$ stands for the unit ball of $\R^N$ and $f$ is a locally 
Lipschitz function which is positive for large values and 
verifies the Keller-Osserman condition. Under an additional 
hypothesis on the asymptotic behavior of $f$ we show that all 
solutions of the above problem are radially symmetric and radially 
increasing. Our condition is sharp enough to generalize several  
results in previous literature.
\end{abstract}

\maketitle

%%%%%%%%%%%%%1.Introducci\'{o}n

\section{Introduction}
\setcounter{section}{1}
\setcounter{equation}{0}

In this paper we are addressing the question of radial symmetry of solutions of 
the boundary blow-up problem 
\begin{equation}\label{problema}
\left\{
\begin{array}{ll}
\De u = f(u) & \hbox{in } B\\
\ \ u=+\infty & \hbox{on }\p B.
\end{array}
\right.
\end{equation}
Here $B$ stands for the unit ball of $\R^N$, $B:=\{x\in \R^N:\ |x|<1\}$, and 
$f$ is a locally Lipschitz nonlinearity. The boundary condition is understood in 
the sense $u(x)\to +\infty$ whenever $x\to \p B$. 

Our interest in this problem comes from the symmetry results 
obtained for the semilinear Dirichlet problem
\begin{equation}\label{eq-Dirichlet}
\left\{
\begin{array}{ll}
-\De u = f(u) & \hbox{in } B\\
\ \ u=0 & \hbox{on }\p B.
\end{array}
\right.
\end{equation}
It is very well-known since the classical reference \cite{GNN} that positive solutions of this problem 
are radially symmetric and radially decreasing. Indeed, this property also holds for some more 
general elliptic problems related to \eqref{eq-Dirichlet}.

\medskip

In light of this achievement, it is to be expected that \eqref{problema} enjoys a similar 
property. As a matter of fact, it has been conjectured by H. Brezis that all solutions of \eqref{problema} are 
radially symmetric and radially increasing (see \cite{PV}). However, at the best of our knowledge 
only partial results in this direction have been obtained. Let us mention \cite{MRW}, \cite{PV} 
and \cite{CD}, where solutions have been shown to be radially symmetric and radially increasing 
under some restrictions on the behavior of $f$ at infinity. 

It is to be noted that the singularity of the solutions of \eqref{problema} near $\p B$ introduces  
additional difficulties when trying to implement the standard techniques that do work for 
\eqref{eq-Dirichlet}. For instance, the method of moving planes, which is the main tool used 
to achieve most of the symmetry results (cf. \cite{A}, \cite{S}, \cite{GNN}, \cite{BN}) 
is not directly applicable unless some information on the behavior of solutions near the boundary 
is available. Unfortunately, the only known way for the moment to obtain this information 
is at the expense of requiring further hypotheses on $f$ for large values.

\medskip

Let us next give a precise statement of our findings. We will assume that the function $f$ 
is positive for large values, in the sense that 
$$
\hbox{there exists } a>0 \hbox{ such that } f(a)>0 \hbox{ and } f(t)\ge 0 \hbox{ for }
t > a.
$$
\noindent It is then well-known that problem \eqref{problema} admits solutions if and only if 
the Keller-Osserman condition 
\begin{equation}\label{KO}
\int_{t_0}^\infty \frac{ds}{\sqrt{F(s)}} <+\infty
\end{equation}
holds, where $F(t)=\int_0^t f(s) ds$. See the classical references \cite{K} and \cite{O} when 
$f$ is monotone, and the more recent paper \cite{DDGR}, where the monotonicity condition is dispensed with.
Thus we will also assume that \eqref{KO} holds.

Our main restriction in the study of problem \eqref{problema} is the following condition: there exist 
$p>1$ and $K>0$ such that 
\begin{equation}\label{hipo-f}
f(t)+K t^p \hbox{ is nondecreasing for large } t.
\end{equation}
This hypothesis is to be complemented with the following one:
\begin{equation}\label{hipo-segunda}
\lim_{t\to +\infty} \frac{t^\frac{p-1}{2} \phi(t)}{\sqrt{F(t)}}=0,
\end{equation}
where
\begin{equation}\label{def-phi}
\phi(t)=\int_t^\infty \frac{ds}{F(s)}.
\end{equation}
It is interesting to remark that the case $p=1$ in \eqref{hipo-f} 
was already considered in \cite{CD}, where the symmetry of solutions of \eqref{problema} 
was obtained with no further restrictions on $f$. Observe that in this case 
\eqref{hipo-segunda} automatically holds, since $\lim_{t\to +\infty} \phi(t)=0$ and 
$\lim_{t\to +\infty} F(t)=+\infty$.

We can show the following:

\begin{teorema}\label{th-simetria}
Assume $f$ is locally Lipschitz, positive for large values and verifies the Keller-Osserman condition \eqref{KO}. 
If moreover $f$ verifies \eqref{hipo-f} for some $K>0$ and $p>1$ and \eqref{hipo-segunda}, then 
every solution $u\in C^2(B)$ of \eqref{problema} is radially symmetric and radially increasing.
\end{teorema}

\medskip

It is important to note that condition \eqref{hipo-segunda} is automatically satisfied whenever 
\eqref{hipo-f} holds with $p=5$ and $f$ verifies the Keller-Osserman condition (cf. the proof of 
Corollary \ref{th-simetria-2} below). Therefore, we obtain a 
symmetry result which substantially generalizes those in \cite{MRW}, \cite{PV} and \cite{CD}. 
In these works, the hypotheses imposed on $f$ always imply \eqref{hipo-f} with $p=1$.

\begin{coro}\label{th-simetria-2}
Assume $f$ is locally Lipschitz, positive for large values and verifies the Keller-Osserman condition 
\eqref{KO}. If moreover there exists $K>0$ such that 
$$
f(t)+K t^5 \hbox{ is nondecreasing for large } t,
$$
\noindent then every solution $u\in C^2(B)$ of \eqref{problema} is radially symmetric and radially increasing.
\end{coro}

\bigskip

One of the main advantages of our approach is that it enables us to deal with nonlinearities 
whose derivative highly oscillates at infinity. As an example, we consider the class of problems
\begin{equation}\label{problema-ejemplo}
\left\{
\begin{array}{ll}
\De u = u^q (1+\sin u) & \hbox{in } B\\
\ \ u=+\infty & \hbox{on }\p B,
\end{array}
\right.
\end{equation}
where $q>1$, which have been introduced in \cite{DDGR}. As far as we know, the symmetry 
results available up to now do not apply to problem \eqref{problema-ejemplo}. 
But, as a consequence of Theorem \ref{th-simetria}, we have

\begin{coro}\label{coro-ejemplo}
Assume $q>1$. Then every solution of problem \eqref{problema-ejemplo} is radially symmetric and 
radially increasing.
\end{coro}

As a matter of fact, we will give also an alternative proof of this corollary, whose idea 
can be applied to give symmetry for some more general problems. See Remark \ref{rem-general} 
in Section 4.

\bigskip

To conclude the Introduction, let us briefly describe our technique of proof. We employ 
the method of moving planes. But instead of comparing the normal derivative of a solution 
of \eqref{problema} with the tangential ones, as in \cite{PV} or \cite{CD}, we use a maximum 
principle in narrow domains in the spirit of \cite{BN}. However, it is worthy of mention that 
the problems we obtain with this approach are very singular near $\p B$, so that the 
maximum principle in narrow domains is not expected to be true in general. Thus it is important 
to apply it only for domains which become \emph{really narrow} when approaching the boundary. 
It is here where condition \eqref{hipo-segunda} comes into play. Its role is to make 
compatible the singularity in the equation with the width of the domain. See precise details 
in Section 3. 

Finally, we would also like to say that the power function in condition \eqref{hipo-f} is only 
a convenience, and it can be replaced by more general functions. For instance, if we assume 
the existence of $\al>0$ and $K>0$ such that $f(t) + K e^{\al t}$ is nondecreasing for large 
$t$ and 
$$
\lim_{t\to +\infty} \frac{e^\frac{\al t}{2}\phi(t)}{\sqrt{F(t)}}=0,
$$
then every positive solution of \eqref{problema} is radially symmetric and radially increasing.
This makes it possible to obtain symmetry of solutions for some very oscillating problems like 
$$
\left\{
\begin{array}{ll}
\De u = e^{\al u} (1+\sin u) & \hbox{in } B\\
\ \ u=+\infty & \hbox{on }\p B,
\end{array}
\right.
$$
where $\al>0$. The interested reader can work out the details, which we omit for brevity.

\medskip

The rest of the paper is organized as follows: in Section 2 we will give some preliminaries 
dealing with estimates of solutions of \eqref{problema} near the boundary of $B$. Section 3 is 
dedicated to the proof of a maximum principle for singular problems in narrow domains, while in 
Section 4 we perform the proof of the symmetry results, Theorem \ref{th-simetria} and 
Corollaries \ref{th-simetria-2} and \ref{coro-ejemplo}.

\bigskip

\section{Preliminaries}
\setcounter{section}{2}
\setcounter{equation}{0}

In this section we collect some preliminary results related to problem \eqref{problema}. 
We are mainly interested in estimates for all possible solutions of \eqref{problema} 
near $\p B$. In this regard, the function
\begin{equation}\label{def-psi}
\psi(t)= \frac{1}{\sqrt{2}} \int_t^\infty \frac{ds}{\sqrt{F(s)}}
\end{equation}
which is well-defined for large $t$ because of the positivity of $f$ for large values and 
the Keller-Osserman condition \eqref{KO}, is known to play a prominent role. Also the function 
$\phi$ defined by \eqref{def-phi} in the introduction will turn out to be relevant. In 
order that our estimates can be posed globally, we first define the function for all 
values of $t$ by choosing a large enough $t_0$ and setting
\begin{equation}\label{def-phi-2}
\phi(t)=\left\{
\begin{array}{ll}
\ds \int_t^\infty \frac{ds}{F(s)} & \hbox{if } t>t_0,\\[1pc]
\ds \int_{t_0}^\infty \frac{ds}{F(s)} & \hbox{if } t \le t_0
\end{array}
\right.
\end{equation}
The first result related to boundary behavior of solutions is Theorem 
1.5 in \cite{DDGR}. We are setting as usual in this context $d(x)=\hbox{dist}(x,\p B)=1-|x|$. 

\begin{lema}\label{lema-comp-asint}
Let $u\in C^2(B)$ be a solution of \eqref{problema}. Then 
$$
\lim_{x\to \p B} \frac{\psi(u(x))}{d(x)}=1.
$$
\end{lema}

\bigskip

As for the behavior of the derivatives of solutions of \eqref{problema}, there is no general 
information available except for nonlinearifies $f$ subject to further restrictions. However, 
when it comes to radially symmetric solutions $U$, it can be easily proved that 
\begin{equation}\label{comp-derivada}
U'(r) \sim 2 \sqrt{F(U(r))} \quad \hbox{as } r\to 1
\end{equation}
(cf. for instance the proof of Proposition 3.1 in \cite{DDGR}).

Finally, let us include a result which somehow refines the boundary behavior given in 
Lemma \ref{lema-comp-asint}. It essentially gives information on the growth of an 
arbitrary solution $u$ of \eqref{problema}, by comparing it with a radially symmetric 
solution, and it will be an important ingredient in our proofs of symmetry. See
Lemma 2.4 and Theorem 1.1 in \cite{CDG}.

\begin{lema}\label{lema-exist-radial}
Assume $u\in C^2(B)$ is a solution of \eqref{problema}. Then there exists a radially 
symmetric solution $U\in C^2(B)$ of \eqref{problema} such that $u\le U$ in $B$.
Moreover, there exists $C>0$ such that 
$$
U(x)-u(x) \le C \phi(U(x)), \quad x\in B,
$$
where $\phi$ is given in \eqref{def-phi-2}.
\end{lema}

\bigskip

\section{A maximum principle in narrow domains}
\setcounter{section}{3}
\setcounter{equation}{0}

This section is devoted to consider the fundamental tool in the proof of our symmetry results, 
a maximum principle for singular problems in suitable narrow domains. Before we can 
state it, we need to introduce some notation related to the method of moving planes, which 
is rather standard by now (cf. \cite{GNN}, \cite{BN}). For $\la \in (0,1)$, set 
$$
\begin{array}{l}
\ds \Sigma_\la:=\{x\in B: \ x_1 > \la\}\\[0.25pc]
\ds T_\la :=\{x\in B: \ x_1 = \la\}\\[0.25pc]
x_\la=(2\la-x_1,x'), \quad \hbox{where } x=(x_1,x')\in B.
\end{array}
$$
Moreover, if $u$ is any function defined on $B$, we also set
$$
u_\la(x)=u(x_\la), \quad x\in \Sigma_\la.
$$
Next let $p>1$, $C_0>0$ and take a radially symmetric solution $U$ of \eqref{problema}. 
If $D\subset \Sigma_\la$ is an arbitrary domain, we are interested in 
determining whether the maximum principle holds for the problem
\begin{equation}\label{eq-pm}
\left\{
\begin{array}{ll}
\De u + C_0 U_\la(x)^\frac{p-1}{2} u \le 0 & \hbox{in } D\\[0.25pc]
\ \ \ds \liminf_{x\to \p D} u \ge 0.
\end{array}
\right.
\end{equation}
Notice that $U_\la(x)^\frac{p-1}{2}$ can be very singular when $x_\la$ approaches 
$\p B$, hence the maximum principle is not expected to hold in general for 
problem \eqref{eq-pm}, even 
if $\la$ is close enough to 1 (see Remark \ref{rem-maxp} below). However, it can be 
recovered when the domain $D$ is suitably narrow near $\p B$. To make this more 
precise, take $x_0\in \p B$ and let $\mathcal{U}$ be a neighborhood of $x_0$. 
We will assume that $D$ is contained in a sort of `lentil' of the form 
$D\subset \{x\in B: \ \la<x_1<\la +H(x)\}$, where 
$$
H(x)= \frac{\phi(U_\la(x))}{\sqrt{F(U_\la(x))}}.
$$
Observe that $H(x)\to 0$ as $x \to T_\la \cap \p B$, therefore $D$ is really narrow when 
approaching the boundary. This implies that, if $x_0\not \in T_\la$ and $x_0\ne e_1$ 
then $D\cap \mathcal{U}$ would be empty for a small enough neighborhood. Thus 
only the cases $x_0\in T_\la \cap \p B$ and $x_0 =e_1$ are meaningful.

The validity of condition \eqref{hipo-segunda} is 
essential in our next result.

\begin{lema}\label{lema-MP}
Assume $f$ is positive for large values and verifies the Keller-Osserman condition \eqref{KO} and 
hypothesis \eqref{hipo-segunda}. Then for $x_0=e_1$ or $x_0\in T_\la \cap \p B$, there exists a neighborhood 
$\mathcal{U}$ of $x_0$ such that if $D\subset \mathcal{U}\cap B\cap \{x:\ \la<x_1<\la+H(x)\}$ for some 
$\la\in (0,1)$, and $u\in C^2(D)$ verifies \eqref{eq-pm}, we have $u\ge 0$ in $D$.
\end{lema}

\begin{proof}
We begin by remarking that the choice of $x_0$ implies the following: for every neighborhood $\mathcal{V}$ 
of $x_0$, there exists another neighborhood $\mathcal{U}$ such that $x_\la \in \mathcal{V}$ 
whenever $x\in \mathcal{U}$. Thus we are only dealing in what follows with points $x$ such 
that $x_\la$ is near the boundary.

The first step is to look for a function $\omega$ verifying $\De \omega + C_0 U_\la(x)^{p-1} \omega < 0$ 
in $D$. We choose
$$
\omega(x)= \cos (\mu U_\la(x)^\frac{p-1}{2} (x_1-\la)),
$$
where $\mu>0$ will be selected later on. It clear that
\begin{equation}\label{eq-otra-estim}
\begin{array}{rl}
\De \omega \hspace{-2mm} & = \ds -\mu^2 \cos(\mu U_\la(x)^\frac{p-1}{2} (x_1-\la)) |\nabla U_\la(x)^\frac{p-1}{2} 
(x_1-\la)|^2\\[0.25pc]
& \quad  \ds -\mu \sin (\mu U_\la(x)^\frac{p-1}{2} (x_1-\la)) \De  (U_\la(x)^\frac{p-1}{2} (x_1-\la)).
\end{array}
\end{equation}
Let us estimate each of these two terms. We have
\begin{align}\label{eq-est-1}
|\nabla U_\la(x)^\frac{p-1}{2} (x_1-\la)| & =
\left|\frac{p-1}{2} U_\la(x)^\frac{p-3}{2} \nabla U_\la(x) (x_1-\la) + U_\la(x)^\frac{p-1}{2} e_1\right| \nonumber \\
& \ge U_\la(x)^\frac{p-1}{2} - \frac{p-1}{2} U_\la(x)^\frac{p-3}{2}|\nabla U_\la(x)| (x_1-\la)\\
& \ge U_\la(x)^\frac{p-1}{2} - \frac{p-1}{2} U_\la(x)^\frac{p-3}{2}|\nabla U(x_\la)| H(x). \nonumber
\end{align}
Using \eqref{comp-derivada} we see that $|\nabla U(x_\la)|\ge C\sqrt{F(U_\la(x))}$ 
when $x_\la$ is in a neighborhood of $\p B$. It then follows from \eqref{eq-est-1} and the definition of 
$H(x)$ that 
\begin{equation}\label{eq-est-final-1}
|\nabla U_\la(x)^\frac{p-1}{2} (x_1-\la)| \ge 
U_\la(x)^\frac{p-1}{2} \left(1-C\frac{\phi(U_\la(x))}{U_\la(x)}\right)
\ge \frac{1}{2} U_\la(x)^\frac{p-1}{2},
\end{equation}
when $x_\la$ is in a neighborhood of $\p B$. 
As for the last term in \eqref{eq-otra-estim}, we obtain 
\begin{align*}
\De  (U_\la(x)^\frac{p-1}{2} (x_1-\la)) & = \frac{p-1}{2} \frac{p-3}{2} U_\la(x)^\frac{p-5}{2} 
|\nabla U_\la(x)|^2 (x_1-\la)\\
& + \frac{p-1}{2} U_\la(x)^\frac{p-3}{2} f(U_\la(x)) (x_1-\la) \\
& +(p-1) U_\la(x)^\frac{p-3}{2} \p_1 U_\la(x).
\end{align*}
Proceeding as above and taking into account that $f$ is nonnegative for large values, 
we deduce
\begin{align}\label{eq-est-final-2}
\nonumber \De  (U_\la(x)^\frac{p-1}{2} (x_1-\la)) & \ge - C U_\la(x)^\frac{p-5}{2} F(U_\la(x)) H(x)\\ 
& \quad -C U_\la(x)^\frac{p-3}{2} \sqrt{F(U_\la(x))}\nonumber \\ 
& = \ds  -C U_\la(x)^\frac{p-3}{2} \sqrt{F(U_\la(x))}\left(1+C \frac{\phi(U_\la(x))}{U_\la(x)}\right)\\
& \ge -C U_\la(x)^\frac{p-3}{2} \sqrt{F(U_\la(x))}. \nonumber
\end{align}
Assume for the moment that 
\begin{equation}\label{eq-requisito}
\mu U_\la(x)^\frac{p-1}{2} (x_1-\la)\le \frac{\pi}{4}.
\end{equation} 
Thus, we obtain from \eqref{eq-otra-estim}, \eqref{eq-est-final-1} and \eqref{eq-est-final-2} that
\begin{align*}
\De \omega + C_0 U_\la(x)^{p-1} \omega & \le (C_0 -C\mu^2) U_\la(x)^{p-1}\cos(\mu U_\la(x)^\frac{p-1}{2} (x_1-\la))\\
& \quad +C \mu U_\la(x)^\frac{p-3}{2} \sqrt{F(U_\la(x))} \sin (\mu U_\la(x)^\frac{p-1}{2} (x_1-\la)).
\end{align*}
Now we choose $\mu$ large enough so that $C_0-C\mu^2 \le -1$, say. Using that 
$\sin z\le z$ for $z>0$, and that the cosine is bounded away from zero, we obtain
\begin{align*}
\De \omega + C_0 U_\la(x)^{p-1} \omega & \le -C U_\la(x)^{p-1}
+C  U_\la(x)^{p-2} \sqrt{F(U_\la(x))}(x_1-\la)\\
& \le -C U_\la(x)^{p-1} +C  U_\la(x)^{p-2} \phi(U_\la(x))\\
& = -C U_\la(x)^{p-1} \left(1- C \frac{\phi(U_\la(x))}{U_\la(x)}\right) < 0,
\end{align*}
if $x_\la$ is close enough to $\p B$. This gives the desired function $\omega$, provided that 
\eqref{eq-requisito} holds. To check that it actually holds, observe that 
$$
\mu U_\la(x)^\frac{p-1}{2} (x_1-\la) \le \mu U_\la(x)^{p-1} \frac{\phi(U_\la(x))}{\sqrt{F(U_\la(x))}} 
\to 0
$$
as $x_\la\to \p B$, because of condition \eqref{hipo-segunda}. Notice also that \eqref{eq-requisito} 
also shows that $\omega$ is bounded away from zero in $\overline{D}$. 

To conclude the proof we argue in a more or less standard way. Assume $u$ is as in the statement of 
the lemma and let $z=u/\omega$. Then 
$$
\omega \De z  + 2 \nabla z\nabla \omega + z(\Delta \omega+CU_\la(x)^{p-1} \omega)\le 0 
\quad \hbox{in } D,
$$
while $\liminf_{x\to \p D} z(x)\ge 0$. If $z$ is negative somewhere then it has to achieve a negative 
minimum at some point $x_0\in D$. It follows that 
$$
z(x_0) (\De \omega(x_0) + C U_\la(x_0)^{p-1} \omega(x_0)) \le 0,
$$
a contradiction. Thus $z\ge 0$ in $D$, therefore $u\ge 0$ in $D$, as we wanted to show.
\end{proof}

\medskip

\begin{obs}\label{rem-maxp}{\rm 
With the same kind of ideas, the ``critical" problem
\begin{equation}\label{eq-remark-pm}
\left\{
\begin{array}{ll}
\De u + C_0 d(x)^{-2} u \le 0 & \hbox{in } D\\[0.25pc]
\ds \liminf_{x\to \p D} u \ge 0,
\end{array}
\right.
\end{equation}
where $C_0>0$, can be analyzed. We assume that $D$ is a subdomain of a smooth bounded domain $\Om$, and 
moreover $D$ is contained for instance in a slab of the form $0<x_1<H(x)$, where the function $H$ 
verifies 
$$
\lim_{x\to \p \Om} \frac{H(x)}{d(x)}=0
$$
with $d(x)=\text{dist}(x,\p\Om)$. Then a maximum principle in the spirit of Lemma \ref{lema-MP} 
for problem \eqref{eq-remark-pm} holds.

\medskip

Problem \eqref{eq-remark-pm} is critical in the sense that the maximum principle 
in narrow domains does not hold in general, without further requiring a smallness assumption 
on the set $D$. To see this, it is enough to consider the one-dimensional version 
\begin{equation}\label{eq-remark-unidim}
u'' + C_0 x^{-2} u \le 0 \quad \hbox{in } (0,\delta).
\end{equation}
When $C_0>\frac{1}{4}$, the function 
$$
u(x) = x^\frac{1}{2} \cos\left( \frac{\sqrt{4C_0-1}}{2} \log x\right)
$$
verifies inequality \eqref{eq-remark-unidim} and has infinitely many zeros accumulating at $x=0$, 
hence the maximum principle is not valid. With a standard change of variables, it can 
be seen that the same situation arises in the radially symmetric version of 
\eqref{eq-remark-unidim}, which coincides with \eqref{eq-remark-pm} when $\Om$ is a ball.
}\end{obs}

\bigskip

\section{Proof of the symmetry results}
\setcounter{section}{4}
\setcounter{equation}{0}

In this final section we will give our proofs of Theorem \ref{th-simetria} and 
Corollaries \ref{th-simetria-2} and \ref{coro-ejemplo}. The proof of Theorem \ref{th-simetria} 
follows with the use of moving planes, while that of Corollary \ref{coro-ejemplo} is a simple 
consequence. In spite of this, we are providing a second proof of the corollary 
which can also be generalized to deal with some more general nonlinearities. 

\medskip

\begin{proof}[Proof of Theorem \ref{th-simetria}]
In order to make the proof clearer, we will divide it in several steps.
We adopt the notation introduced in Section 3. 

\medskip

\noindent {\it Step 1.} $u\ge u_\la$ in $\Sigma_\la$ if $\la$ is close to 1. 

\smallskip

Assume this is not true, so that the set
$$
D_\la:=\{ x\in \Sigma_\la:\ u(x)<u_\la(x)\}
$$
is nonempty and open. By Lemma \ref{lema-exist-radial}, there exists a radially symmetric solution 
$U$ of \eqref{problema} such that $u\le U$ in $B$. Since $U$ is radially increasing, it follows 
that $U\ge U_\la$ for every $\la$. Thus if $x\in D_\la$ we have for some $C>0$
\begin{align*}
0 & <u_\la(x)-u(x) \le U_\la(x) - U(x) + C \phi(U(x))\\
& \le U_\la(x) - U(x) + C \phi(U_\la(x))\\
& =2\p_1 U(\xi) (\la-x_1) + C \phi(U_\la(x)),
\end{align*}
where $\xi=\xi(x)$ is a point in the segment $[x_\la,x]$ and we have used the mean value theorem. 
We will use throughout the notation $\p_1$ for the partial derivative with respect to the 
variable $x_1$.

Notice that $(r^{N-1}U'(r))'\ge 0$ for $r$ close to 1, while $U'(r)>0$ when $r>0$. 
Therefore, if $\la$ is close enough to 1, 
using that $|\xi| \ge |x_\la|$ and \eqref{comp-derivada} we obtain
\begin{align*}
\p_1 U(\xi) & = U' (\xi)\frac{\xi_1}{|\xi|} \ge C U'(\xi) \ge \left( \frac{|x_\la|}{|\xi|}\right)^{N-1} U'(x_\la)\\
& \ge C U'(x_\la) \ge C \sqrt{F(U_\la(x))}.
\end{align*}
Thus for every $x\in D_\la$ we have
$$
0< -C(x_1-\la)\sqrt{F(U_\la)} + C\phi(U_\la).
$$
This means that the set $D_\la$ is contained in the slab $\la < x_1<\la + H(x)$,
where 
$$
H(x):= C\frac{\phi(U_\la(x))}{\sqrt{F(U_\la(x))}}.
$$
On the other hand, if we denote $w_\la=u-u_\la$ and use hypothesis \eqref{hipo-f} on $f$ we 
see that 
$$
\Delta w_\la + K\left(\frac{u^p-u_\la^p}{u-u_\la} \right) w_\la \le 0 \qquad \hbox{in } D_\la.
$$
Again because of the mean value theorem, there exists $\eta=\eta(x)$ such that 
$u(x)<\eta<u_\la(x)$ and 
$$
\frac{u^p-u_\la^p}{u-u_\la}= p\eta^{p-1} \le p u_\la^{p-1} \le p U_\la^{p-1}.
$$
Thus
$$
\De w_\la + C U_\la^{p-1} w_\la \le 0 \quad \hbox{in } D_\la.
$$
Moreover, $\liminf_{x\to \p D_\la} w_\la \ge 0$. This is clear when $x$ approaches 
points on $\p D_\la\cap B$. When $x$ approaches a point $x_0 \in \p D_\la \cap \p B$, 
recalling that $w_\la(x) > U(x)- U_\la(x) - C \phi(U_\la(x))$, we have
$$
\liminf_{x\to x_0} w_\la(x) \ge \liminf_{x\to x_0}  -C \phi(U_\la(x)) = 0.
$$
Thus, we have shown that
$$
\left\{
\begin{array}{ll}
\De w_\la + C U_\la^{p-1} w_\la \le 0 & \hbox{in } D_\la\\[0.25pc]
\ds \liminf_{x\to \p D_\la} w_\la \ge 0.
\end{array}
\right.
$$
Observe that hypothesis \eqref{hipo-segunda} and Lemma \ref{lema-MP} give $w_\la\ge 0$ in $D_\la$ if 
$\la$ is close enough to 1, a contradiction. This contradiction shows that $D_\la$ is 
empty for $\la$ close enough to 1, therefore $u\ge u_\la$ in $\Sigma_\la$ if $\la$ is 
close enough to 1.

\medskip

\noindent {\it Step 2}. 
$$
\la^*=\inf\{ \la>0:\ u\ge u_\mu \hbox{ in } \Sigma_\mu, \hbox{ for every } \mu\in (\la,1)\}=0.
$$

\smallskip

Assume for a contradiction that $\la^*>0$. By continuity, we know that $u\ge u_{\la^*}$ in $\Sigma_{\la^*}$. 
Since we can write
$$
-\De (u-u_{\la^*})+a(x) (u-u_{\la^*})=0 \qquad 
\hbox{in }\Sigma_{\la^*}
$$
where $a(x)= \frac{f(u)-f(u_{\la^*})}{u-u_{\la^*}}$ is locally bounded in $\Sigma_{\la^*}$, 
we are in a position to apply the strong maximum principle to deduce that either $u\equiv u_{\la^*}$ 
or $u>u_{\la^*}$ in $\Sigma_{\la^*}$. Since we are assuming $\la^*>0$, the first option 
is not possible because $u=+\infty$ on the portion of $\p \Sigma_{\la^*}$ which touches the 
boundary of $B$, while $u_{\la^*}$ is finite there. Thus $u>u_{\la^*}$ in $\Sigma_{\la^*}$ 
and Hopf's principle gives $\p_1 u >\p_1 u_{\la_*}$ on $T_{\la^*}$, 
which readily implies $\p_1 u >0$ on $T_{\la^*}$. Therefore
\begin{equation}\label{eq-contra}
u>u_{\la^*} \quad \hbox{in }\Sigma_{\la^*} \qquad \hbox{and} \qquad 
\p_1 u >0 \quad \hbox{on } T_{\la^*}.
\end{equation}
On the other hand, by the very definition of $\la^*$ there exist sequences of positive 
numbers $\la_n \uparrow \la^*$ and points $x_n\in \Sigma_{\la_n}$ such that $u(x_n)<u_{\la_n}(x_n)$. We claim that 
the points $x_n$ can be taken outside a neighborhood of $T_{\la^*} \cap \p B$. 

Otherwise we would have that the set $D_{\la_n}:=\{x\in B:\ u(x)<u_{\la_n}(x)\}$ is 
contained in a neighborhood of $T_{\la^*}\cap \p B$ as small as we desire. We can then argue as 
in Step 1 to show that $D_{\la_n}$ is empty, a contradiction. 

Thus we may assume that $x_n\to x_0 \in \overline{\Sigma}_{\la^*} \setminus (T_{\la^*} 
\cap \p B)$. There are three cases to consider:

\begin{itemize}

\item[(a)] $x_0\in \Sigma_{\la^*}$; this would imply by continuity of $u$ that $u(x_0)\le u_{\la^*}(x_0)$, 
which is impossible by \eqref{eq-contra}.

\item[(b)] $x_0\in \p B$; this is also impossible, since then $u-u_{\la_n}$ would be 
positive in a neighborhood of $x_0$ in contradiction with the choice of $x_n$. 

\item[(c)] $x_0\in T_{\la^*}\cap B$; let $\eta_n$ be the projection of $x_n$ on $T_{\la_n}$. 
Since $u(x_n)-u_{\la_n}(x_n)<0$ and $u(\eta_n)-u_{\la_n}(\eta_n)=0$, there 
exists a point $\xi_n$ in the segment $[x_n,\eta_n]$ such that $\p_1(u(\xi_n)-u_{\la_n}(\xi_n))< 0$. 
It is clear that $\xi_n\to x_0$, hence we deduce $\p_1(u(x_0)-u_{\la^*}(x_0))\le 0$, 
contradicting \eqref{eq-contra}.

\end{itemize}

The contradiction reached in all cases shows that $\la^*>0$ is impossible, hence $\la^*=0$, 
as claimed.

\medskip

\noindent {\it Step 3}. Completion of the proof. 

\smallskip

By Step 2 we deduce that 
$$
u(x_1,x')\ge u(-x_1,x') \quad \hbox{when } x_1>0
$$
and
$$
\p_1 u >0 \quad \hbox{if } x_1>0.
$$
Since we can replace $x_1$ by $-x_1$, it follows that $u$ is symmetric with respect to the 
direction $x_1$. Finally, this direction is arbitrary so it follows that $u$ is radially symmetric and 
$\p_r u >0$ in $B\setminus \{0\}$. This concludes the proof.
\end{proof}

\bigskip

\begin{proof}[Proof of Corollary \ref{th-simetria-2}]
We only need to show that condition \eqref{hipo-segunda} is verified when $p=5$. 
Observe that $F$ is nondecreasing for large values, so that it is immediate that 
$$
\phi(t) \le \sqrt{\frac{2}{F(t)}} \psi(t).
$$
Hence since $\lim_{t\to +\infty} \psi(t)=0$, there exists a positive constant 
$C$ such that 
$$
\frac{t^2 \phi(t)}{\sqrt{F(t)}} \le C\frac{t^2}{F(t)}
$$
for large $t$. Then \eqref{hipo-segunda} follows because the Keller-Osserman 
condition \eqref{KO} implies $F(t)/t^2\to +\infty$. The proof is concluded.
\end{proof}

\bigskip

\begin{proof}[First proof of Corollary \ref{coro-ejemplo}]
The proof is immediate from Theorem \ref{th-simetria}. First of all, observe that 
$f'(t) \ge - t^q$ for positive $t$, therefore $f(t)+\frac{1}{q+1} t^{q+1}$ is 
nondecreasing for positive $t$. It remains to show that \eqref{hipo-segunda} holds 
with $p=q+1$. Observe that $F(t)\sim \frac{1}{q+1} t^{q+1}$ as $t\to +\infty$, 
which implies $\phi(t)\sim \frac{q+1}{q}t^{-q}$ as $t\to +\infty$. Hence
$$
\frac{t^\frac{q}{2} \phi(t)}{\sqrt{F(t)}} \sim 
C t^{-\frac{2q+1}{2}} \to 0,
$$
and we conclude using Theorem \ref{th-simetria}.
\end{proof}

\bigskip

\begin{proof}[Second proof of Corollary \ref{coro-ejemplo}]
Let $u$ be a solution of \eqref{problema-ejemplo}, and $U$ be the 
radially symmetric solution given in Lemma \ref{lema-exist-radial}. 
By Lemma \ref{lema-comp-asint} we have 
$$
U(x) \sim C d(x)^{-\frac{2}{q-1}} \quad \hbox{as } d(x)\to 0.
$$
Since $\phi(t)\sim C t^{-q}$ as $t\to +\infty$, we deduce from Lemma 
\ref{lema-exist-radial} that 
$$
|u(x)-U(x)| \le C d(x)^\frac{2q}{q-1}, \quad x\in B.
$$
Moreover, using the mean value theorem we also see that
\begin{align*}
|\Delta (u-U)(x)| & =|f(u(x))-f(U(x))| \\
& \le C U(x)^q |u(x)-U(x)| \le C, \quad x\in B.
\end{align*}
We can use standard regularity (see for example Chapter 4 in \cite{GT} or an 
explicit statement in Lemma 12 of \cite{GM}) to get
$$
\| \nabla (u-U)\|_{L^\infty(B)} \le C \left( \left \| \frac{u-U}{d} \right\|_{L^\infty(B)}  
+ \| d\Delta (u-U)\|_{L^\infty(B)} \right) \le C.
$$
Denoting by $\nabla_T$ the tangential gradient and by $\p_r$ the radial derivative, 
we have $\nabla_T U=0$ and $\p_r U \to +\infty$ as $x\to \p B$ (cf. \eqref{comp-derivada}). 
In particular, we have for the tangential and radial derivatives of the solution $u$:
$$
\begin{array}{l}
|\nabla_T u|\le C\\[0.5pc]
\p_r u \to +\infty \quad \hbox{as } x\to \p B.
\end{array}
$$
We can then use Theorem 2.1 in \cite{PV} to deduce that $u$ is radially symmetric and radially 
increasing.
\end{proof}

\bigskip

\begin{obs}\label{rem-general} {\rm The approach followed in the second proof of 
Corollary \ref{coro-ejemplo} can be used to obtain a slightly more general result. 
Indeed, assume $f$ is positive for large values, verifies the Keller-Osserman condition 
\eqref{KO} and $L(t)$ stands for the Lipschitz constant of $f$ in the interval $[t_0,t]$ 
for some fixed $t_0$. If there exists $\g\in \R$ such that 
\begin{align}\label{cond-final}
\nonumber \limsup_{t\to +\infty} \psi(t)^{-\g} \phi(t)<+\infty\\
\limsup_{t\to +\infty} \phi(t) \psi(t)^{2-\g}L(t)<+\infty\\
\nonumber \lim_{t\to +\infty} \psi(t)^{2(1-\g)}F(t) = +\infty,
\end{align}
then every solution $u\in C^2(B)$ of \eqref{problema} is radially symmetric and 
radially increasing.

To see this, observe that $\psi(U)\sim d$ as $d\to 0$ (Lemma \ref{lema-comp-asint}). Therefore 
the first condition in \eqref{cond-final} implies $|u-U|\le C \phi(U) \le d^\g$ in $B$, while 
using the second condition
$$
|d^2 \De (u-U)| \le C d^2 L(U) |u-U| \le C d^2 L(U) \phi(U)\le C d^\g \quad \hbox{in } B.
$$
Thus $|d\nabla(u-U)|\le Cd^\g$ in $B$ by classical regularity. Notice also that 
the third condition in \eqref{cond-final} implies $d^{\g-1} =o(\sqrt{F(U)})=o(U')$ as 
$d\to 0$. Hence
$$
\frac{|\nabla_T u|}{\p_r u}\le C \frac{d^{\g-1}}{U'-Cd^{\g-1}}\to 0
$$
as $d\to 0$, while $\p_r u\to +\infty$ as $d\to 0$. We conclude using Theorem 2.1 in \cite{PV}. 
The nonlinearity in \eqref{problema-ejemplo} verifies condition \eqref{cond-final} with $\gamma=2$.
}\end{obs}

\bigskip

\noindent {\bf Acknowledgements.} Supported by FONDECYT 1150028 (Chile).
J. G-M. is also supported by Ministerio de Econom\'ia y Competitividad under 
grant MTM2014-52822-P (Spain).

%%%%%%%%%%%%%%%%%%%%%%%%%%%%%%%%%%%%%%%%%%%%%%%%%%%%%%%%%%%%%%%%%%%%%
%%%                       references                              %%%
%%%%%%%%%%%%%%%%%%%%%%%%%%%%%%%%%%%%%%%%%%%%%%%%%%%%%%%%%%%%%%%%%%%%%


\begin{thebibliography}{XX}

\bibitem{A} {\sc A. D. Alexandrov}, {\em A characteristic property of the spheres}, Ann.
Mat. Pura Appl. {\bf 58} (1962), 303--354.

\bibitem{BN} {\sc H. Berestycki, L. Nirenberg}, {\em On the method of moving planes and the sliding 
method}. Bol. Soc. Brasil. Mat. (N.S.) {\bf 22} (1991), no. 1, 1--37.

\bibitem{CD} {\sc O. Costin, L. Dupaigne}, {\em Boundary blow-up solutions in the unit ball: asymptotics, uniqueness 
and symmetry}. J. Differential Equations {\bf 249} (2010), no. 4, 931--964.

\bibitem{CDG} {\sc O. Costin, L. Dupaigne, O. Goubet}, {\em Uniqueness of large solutions}. J. Math. Anal. 
Appl. {\bf 395} (2012), no. 2, 806--812. 

\bibitem{DDGR} {\sc S. Dumont, L. Dupaigne, O. Goubet, V. R\v adulescu}, {\em Back to the Keller-Osserman 
condition for boundary blow-up solutions}. Adv. Nonlinear Stud. {\bf 7} (2007), no. 2, 271--298.

\bibitem{GM} {\sc J. Garc\'ia-Meli\'an}, {\em Nondegeneracy and uniqueness for boundary blow-up elliptic 
problems}. J. Differential Equations {\bf 223} (2006), no. 1, 208--227. 

\bibitem{GNN} {\sc B. Gidas, W. M. Ni, L. Nirenberg}, {\em Symmetry and related properties via the maximum 
principle}. Comm. Math. Phys. {\bf 68} (1979), no. 3, 209--243.

\bibitem{GT} {\sc D. Gilbarg, N. S. Trudinger}, ``Elliptic partial differential equations of second order". 
Reprint of the 1998 edition. Classics in Mathematics. Springer-Verlag, Berlin, 2001.

\bibitem{K} {\sc J. B. Keller}, {\em On solutions of $\De u=f(u)$}. Comm. Pure Appl. Math. {\bf 10} (1957), 503--510.

\bibitem{MRW} {\sc P. J. McKenna, W. Reichel, W. Walter}, {\em Symmetry and multiplicity for nonlinear elliptic 
differential equations with boundary blow-up}. Nonlinear Anal. {\bf 28} (1997), no. 7, 1213--1225.

\bibitem{O} {\sc R. Osserman}, {\em On the inequality $\De u \ge f(u)$}. Pacific J. Math. {\bf 7} (1957), 1641--1647. 

\bibitem{PV} {\sc A. Porretta, L. V\'eron}, {\em Symmetry of large solutions of nonlinear elliptic equations in a ball}. 
J. Funct. Anal. {\bf 236} (2006), no. 2, 581--591.

\bibitem{S} {\sc J. Serrin}, {\em A symmetry problem in potential theory}. Arch. Rat. Mech. Anal.
{\bf 43} (1971), 304--318.

\end{thebibliography}
\end{document}